\documentclass[10pt,letterpaper,reqno]{amsart}
\usepackage{enumerate}

\usepackage[all]{xy}
\newdir{|>}{!/4.5pt/\dir{|}
*:(1,-.2)\dir^{>}
*:(1,+.2)\dir_{>}}
\newdir{(>}{{}*!/-8pt/\dir{>}}
\newdir^{((}{{}*!/-5pt/\dir^{(}}
\xyoption{2cell}
\xyoption{curve}
\xyoption{rotate}
\UseAllTwocells

\usepackage{bbm}
\usepackage{mathrsfs}
\usepackage{amsfonts}
\usepackage{bm}
\usepackage{amsmath,amssymb,amsthm,stmaryrd}

\usepackage{hyperref}

\usepackage{bussproofs}

\usepackage[numbers, sort&compress]{natbib}
\usepackage{mathbbol}

\swapnumbers

\newcommand{\nocontentsline}[3]{}
\newcommand{\tocless}[2]{\bgroup\let\addcontentsline=\nocontentsline#1*{#2}\egroup}

\newtheorem{teorema}{Theorem}[section]

\newtheorem*{teoA}{Theorem A}
\newtheorem*{teoB}{Theorem B}
\newtheorem*{teoC}{Theorem C}
\newtheorem*{teoD}{Theorem D}
\newtheorem*{teoE}{Theorem E}
\newtheorem{cor}[teorema]{Corollary}

\newtheorem{lema}[teorema]{Lemma}

\newtheorem{prop}[teorema]{Proposition}
\newtheorem{exem}[teorema]{Example}

\theoremstyle{definition}

\newtheorem*{defii}{Definition}

\newtheorem{obse}[teorema]{Remark}

\newcommand{\overbar}[1]{\mkern 1.5mu\overline{\mkern-1.5mu#1\mkern-1.5mu}\mkern 1.5mu}

\newcommand{\E}{\mathcal{E}}
\newcommand{\D}{\mathcal D}

\DeclareMathOperator{\ob}{Ob}

\newcommand{\teo}{\mathrm{Th}}

\newcommand{\subob}{\mathit{\Omega}}

\DeclareMathOperator{\dso}{\Gamma}
\DeclareMathOperator{\dec}{Dec}

\DeclareMathOperator{\ev}{ev}
\DeclareMathOperator{\sh}{Sh}
\newcommand{\set}{\mathbf{Set}}

\begin{document}
\title[Naive homotopy theories in cartesian closed categories]{Naive homotopy theories in cartesian closed categories}
\author[Ruiz]{Enrique Ruiz Hern\'andez${}^\dagger$}
\address[$\dagger$]{Centro de Investigaci\'on en Teor\'ia de Categor\'ias y sus Aplicaciones, A.C. M\'e\-xi\-co.}
\curraddr{}
\email{e.ruiz-hernandez@cinvcat.org.mx}

\author[Sol\'orzano]{Pedro Sol\'orzano${}^\star$}
\address[$\star$]{Instituto de Matem\'aticas, Universidad Nacional Aut\'onoma de M\'exico, Oaxaca de Ju\'a\-rez, M\'e\-xi\-co.}
\curraddr{}
\email{pedro.solorzano@matem.unam.mx}
\thanks{($\star$) {CONACYT-UNAM} Research fellow.}
\date{\today}
\begin{abstract}
An elementary notion of homotopy can be introduced between arrows in a cartesian closed category $\E$. The input is a finite-product-preserving endofunctor $\Pi_0$ with a natural transformation $p$ from the identity which is surjective on global elements. As expected, the output is a new category $\E_p$ with objects the same objects as $\E$. 

Further assumptions on $\E$ provide a finer description of $\E_p$ that relates it to the classical homotopy theory where $\Pi_0$ could be interpreted as  the ``path-connected components'' functor on convenient categories of topological spaces.  In particular, if $\E$ is a 2-value topos the supports of which split and is furthermore assumed to be precohesive over a boolean base, then the passage from $\E$ to $\E_p$ is naturally described in terms of explicit homotopies---as is the internal notion of contractible space. 
\end{abstract}
\subjclass[2020]{Primary 18B25; Secondary 55U40}
\keywords{Topos, connectedness, precohesion}

\maketitle

\tocless\section{Motivation}
\citet{MR3288694} investigates the properties of a space $T$ in order for it to be an unparameterized unit of time in the context of the foundations for Differential Synthetic Geometry. It should be the case that for discrete spaces 
\begin{equation}\label{E:nonbecoming} A^T\cong A,\end{equation}
 i.e. no motion is posible. It ought to be a space so small as to contain only one point yet far from isomorphic to the singleton. In particular, Lawvere considers a reflector $\pi_0$ that preserves finite products and the codomain of which is to consist of spaces $A$ of non-becoming.  And in general, for any other space $X$,  $\pi_0(X^T)\cong\pi_0(X)$ if and only if the pullback $R$ of $ev_0:T^T\to T$ along the unique point $0:1\rightarrow T$ is connected; i.e. $\pi_0(R)=1$.  

Also, \citet{MR2369017} discusses the notion of contractible space in the process of deriving some consequences of his Axiomatic Cohesion. Therein, he suggests the existence of a homotopy category within the definition of extensive quality. A guiding idea is definitely that of homotopies between continuous maps. A topological space is {\em contractible} if its identity map is homotopic to a constant map. 

Intuitively, the category of homotopy classes is thus defined to have the same objects and as arrows functions modulo homotopy, $[X,Y]$. For well-behaved topological spaces, regarding sets of path-connected components as topological spaces with the discrete topology produces an endofunctor $X\mapsto \pi_0(X)$ together with a natural transformation $X\rightarrow \pi_0(X)$. One can obtain the following identification
\begin{equation}\label{E:htoppizero}
[X,Y]\cong \pi_0(Y^X),
\end{equation} 
once $Y^X$ is endowed with a canonical topology (e.g. the compact-open topology on sufficiently nice spaces).  Conversely, one could begin with the natural transformation $1\Rightarrow \pi_0$, and use \eqref{E:htoppizero} as the definition of the equivalence classes of arrows that are {\em homotopic} in some abstract sense.  Marmolejo and Menni \cite{MR3654357} provide a way of formalizing this.

\tocless\section{Main Results}

The main purpose of this report is to propose another method which has the advantage of not requiring to look at enriched categories (cf. \cite{MR3654357}). Starting from a cartesian closed category (CCC) the name of an arrow $f:X\to Y$ is the transpose `$f$'$:1\to Y^X$ of the arrow $f\circ\pi_X:1\times X\to Y$. It provides a way to internalize the arrows: the exponential $Y^X$ `contains' all the arrows in a natural way and there are internal compositions $c: Z^Y\times Y^X\rightarrow Z^X$ that satisfy all the usual properties.

Define a {\em homotopy theory} on a cartesian closed category $\E$ to be a natural transformation $p:1_\E\Rightarrow\Pi_0:\E\rightarrow\E$  such the functor $\Pi_0$ preserves products and the function 
\[\E(1,p_X):\E(1,X)\Rightarrow\E(1,\Pi_0(X))\]
 is surjective for every $X\in\E$.  
 
 Two arrows $f,g:X\rightarrow Y$ in $\E$ are $p$-\textit{homotopic} to each other if and only if their names $\text{`$f$',`$g$'}:1\rightarrow Y^X$ satisfy
\[
p_{Y^X}\circ\text{`$f$'}=p_{Y^X}\circ\text{`$g$'}.
\]
Denote this equivalence relation by $\sim_p$.
\begin{teoA}\hypertarget{T:TeorA}  For a cartesian closed category $\E$ with a homotopy theory $p:1_\E\Rightarrow\Pi:\E\rightarrow\E$, there is a \textit{homotopic category} $\E_p$ for $\E$ as follows: $\ob(\E_p):=\ob(\E)$ and
\begin{align*}
\E_p(X,Y) &:=\E(X,Y)/{\sim_p}\\
&\cong\E(1,\Pi_0(Y^X)).
\end{align*}
It follows that $\E_p$ is also cartesian closed; and if $\E$ is further assumed to be extensive and distributive, then so is $\E_p$.
\end{teoA}
Denote by $H_p$ the obvious functor from $\E$ to $\E_p$. The authors \cite{RS2023} show that in a topos $\E$ that satisfies a Nullstellensatz (NS, ``any non-initial object has points'') and the WDQO postulate (``For every object there exists a minimal decidable object $p_X:X\to \Pi(X)$ that uniquely factors arrows to $1+1$''),  $\dec(\E)$ is an exponential ideal of $\E$: there is an adjunction $\Pi\dashv \mathcal I:\E\Rightarrow \dec(\E)$, where $\mathcal I$ is the inclusion functor and $\Pi$ preserves finite products. 
\begin{teoB}\hypertarget{T:TeorB} If $\E$ is a topos satisfying NS and WDQO, then the unit $p:1\Rightarrow\mathcal{I}\Pi$ is a homotopy theory for $\E$ and there is an adjunction
\[\xymatrix{
\E_p\ar@/_1pc/[d]_{q_{!}}\ar@<-1ex>@{}[d]^{\dashv} \\
\dec(\E)\ar@/_1pc/[u]_{q^\ast},
}\]
with $q_{!}H_p=\Pi$ and $q^\ast=H_p\mathcal I$. In other words, that $\dec(\E)$ is an exponential ideal of $\E_p$.
\end{teoB}
The identity functor is clearly a homotopy theory.  In toposes that satisfy NS $!_X:X\rightarrow 1$ is a homotopy theory on $\E$. Moreover, in such toposes, any local operator\footnote{These are also called (Lawvere-Tierney) topologies.} $j:\Omega\to\Omega$ induces a homotopy theory $q_j$ by considering the largest $j$-separable quotient (which always exists, see Appendix \ref{A:EQM}).

In the topological context, being homotopic is given explicitly by way of homotopies, continuous maps from the domain times an interval to the codomain. In general, the relationship $\sim_p$ does not explicitly have such a description. One direction is always true, provided that one substitutes ``an interval'' with ``a connected object''.  This replacement is intuitively natural recalling that $\Pi_0$ is to abstract ``path-connected components''.
\begin{teoC}\hypertarget{T:TeorC}  For any cartesian closed $\E$ with a homotopy theory $p$, two arrows $f, g:X\rightarrow Y$ are $p$-homotopic if there is a connected object $A$ with two global elements $a,b:1\rightarrow A$ and an arrow $h:A\times X\rightarrow Y$ such that the following diagrams commute:
\[\xymatrix{
X\ar[r]_(.4){\langle a!,1\rangle}\ar@/^1pc/[rr]^f & A\times X\ar[r]_(.6)h & Y \\
X\ar[r]^(.4){\langle b!,1\rangle}\ar@/_1pc/[rr]_g & A\times X\ar[r]^(.6)h & Y.
}\]
Furthermore, in the case when $\E$ is a topos satisfying NS and WDQO, the converse holds for the induced homotopy theory. 
\end{teoC}

The following result provides an explicit necessary condition to Lawvere's unparameterized unit of time $T$. For a homotopy theory $p$ on a cartesian closed $\E$ say a pointed object $A$ is {\em $p$-contractible} whenever the identity map $1_A$ is $p$-homotopic to the constant map $a!$ for some point $a:1\rightarrow A$.  

\begin{teoD}\hypertarget{T:TeorD}  
Let $\E$ be a cartesian closed category with a homotopy theory $p$. Let $A\in\E$ and for any $X$ let $\sigma_X:X\to X^A$ the transpose of the projection $X\times A\to X$. The following are equivalent.

\begin{enumerate}
\item  The object $A$ is $p$-contractible. 
\item For every object $X\in\E$, $\Pi_0(\sigma_X):\Pi_0(X)\rightarrow\Pi_0(X^A)$ is an isomorphism.
\end{enumerate}
Moreover, when $\E$ is a topos satisfying NS, WDQO and McLarty's DSO postulate (``There exists a unique decidable subobject for any given object containing all its points.'') and $p$ is the associated homotopy theory, they are also equivalent to the following.
\begin{enumerate}
\setcounter{enumi}{2}
\item For every $X\in\E$, $A^X$ is connected. 
\item The object $A$ has a point and $A^A$ is connected.
\end{enumerate}
\end{teoD}

\citet{MR2369017} defines a contractible object to be one that satisfies the property (3) in the previous theorem. The authors \cite{RS2023} show that for a topos $\E$ that satisfies NS, WDQO and DSO it follows that $\dec(\E)$ is a topos and that $\E$ is precohesive over $\dec(\E)$. Therefore, in said context, an object is Lawvere contractible if and only if it is $p$-contractible. Lawvere uses the definition of contractibility to define the notion of {\em sufficient cohesion} (``Any object can be embedded in a contractible object'') and proves a version of the following result (see also \citet{MR3263283}). 
\begin{teoE}\hypertarget{T:TeorE}  Let $\E$ be a topos satisfying NS, WDQO and DSO. The following statements are equivalent:
\begin{enumerate}
 \item $\E$ is sufficiently cohesive.
 \item The identity map of the subobject classifier is $p$-homotopic to a constant map; i.e. $\Omega$ is contractible.
 \item There is a bipointed connected object.
 \item Any nonempty $\neg\neg$-sheaf in $\E$ is connected.
\end{enumerate}
\end{teoE}
In doing so, he proves that this condition is incompatible with being a {\em quality type} (``having connected components with exactly one point each''). If the Law of Excluded Middle is valid in one's metalogic, one has an actual dichotomy in the setting of precohesion over decidables in the presence of NS:  A non degenerate topos $\E$ that satisfies NS, WDQO and DSO is either sufficiently cohesive or a quality type but not both.

Back to the context of Synthetic Differential Geometry, these techniques provide an explicit verification of \eqref{E:nonbecoming} when $T$ any connected object and $A$ decidable; and of Lawvere's Proposition 1 in \cite{MR3288694}: if there is a retraction from $T^T$ to $R$ then $T$ satisfies property (2) of \hyperlink{T:TeorD}{Theorem D} if and only if $R$ is connected (and in fact contractible)---and thus there $\E$ must be sufficiently cohesive.

\subsection*{Aknowledgements.} The authors acknowledge Omar Antol\'in for his permanent support and for several meaningful conversations which ultimately lead to some of the results of this report.

\setcounter{tocdepth}{1}
\tableofcontents



\section{Elementary observations about cartesian closed categories}
 To set notation and to recall some basic features, let $\E$ be a cartesian closed category, axiomatically endowed with the adjunction
\[(\,\cdot\,)\times A \dashv (\,\cdot\,)^A.\] 
 Denote its evaluation map by $\ev_X^A:X^A\times A\to X$.  For any $A$ fixed, let $\sigma_X^A:X\rightarrow X^A$ denote the transpose of $\pi_X:X\times A\to X$. It is natural in $X$.  In the case $A=1$, $\sigma_X^1:X\rightarrow X^1$ is an isomorphism, and its inverse is
\[\xymatrix{
X^1\ar[r]_(.4){\pi_{X^1}^{-1}} & X^1\times 1\ar[r]_(.63){\ev_X^1} & X.
}\]
If $X,A\in\E$, then $\sigma_X^A:X\rightarrow X^A$ is naturally isomorphic to $X^{!}:X^1\rightarrow X^A$ in the sense that
\begin{equation}\label{E:Xtotheuniquesigma}
X^!\circ\sigma_X^1=\sigma_X^A.
\end{equation}
The name `$f$'$:1\to Y^X$ of an arrow $f:X\to Y$ is the transpose of the arrow $f\circ\pi_X:1\times X\to Y$.  The maps $\sigma$ assign to any point the name of the corresponding constant function: If $a:1\rightarrow A$ is a point of $A\in\E$ and $B$ is arbitrary, then 

\begin{equation}\label{E:aadequalssigmaa}
\sigma_A^B\circ a=\text{`$a\circ!_B$'}.
\end{equation}
For a point $t:1\rightarrow T$ and an object $X\in\E$, define $\ev^t_X:X^T\rightarrow X$ as the composite
\begin{equation}\label{E:evt}
\xymatrix{
X^T\ar[r]^(.4){\langle 1,!\rangle} & X^T\times 1\ar[r]^{1\times t} & X^T\times T\ar[r]^(.6){\ev_X^T} & X.
}
\end{equation}
It is naturally isomorphic to $X^t:X^T\to X^1$:
\begin{equation}
X^t=\sigma_X^1\circ \ev^t_X,
\end{equation}
which justifies $X^t$ being considered as, or even called, evaluation at $t$. Also, since $1=!\circ t$, 
\begin{equation}
\ev^t_X\circ\sigma_X^T=1,
\end{equation}
which proves that $\sigma_X^T$ is split monic. 

The internal composition (see Section 6.3 in \citet{MR1182992}) $c:Z^Y\times Y^X\rightarrow Z^X$ is the transpose of
\[\xymatrix{
(Z^Y\times Y^X)\times X\ar[r]^\cong & Z^Y\times(Y^X\times X)\ar[r]^(.6){1\times\ev} & Z^Y\times Y\ar[r]^(.6)\ev & Z,
}\]
and makes the following diagram commute:
\begin{equation}\label{E:internalcomp}
\begin{gathered}
\xymatrix{
1\ar[dr]^{\text{`$g\circ f$'}}\ar[d]_{\langle\text{`$g$'},\text{`$f$'}\rangle} & \\
Z^Y\times Y^X\ar[r]_(.6)c & Z^X
}
\end{gathered}
\end{equation}
for every arrow $f:X\rightarrow Y$ and every arrow $g:Y\rightarrow Z$ in $\E$. If $f:X\rightarrow Y$ and $g:Y\rightarrow Z$ are arrows in $\E$, then 
\begin{equation}
g^X\circ\text{`$f$'}=\text{`$g\circ f$'},
\end{equation}
and
\begin{equation}
Z^f\circ\text{`$g$'}=\text{`$g\circ f$'}. 
\end{equation}
If $\E$ is distributive, then there is a natural isomorphism $\alpha:Z^{X+Y}\to Z^X\times Z^Y$ for every $X,Y,Z\in\E$ such that 
\begin{equation}\label{E:distributiveimpliestothesum}
\alpha\circ\hat{f}=\langle\widehat{f_1},\widehat{f_2}\rangle
\end{equation}
for any $f:A\times(X+Y)\rightarrow Z$, where the arrows $f_1$ and $f_2$ are defined by the following commutative diagram:
\[\xymatrix{
A\times X+A\times Y\ar[dr]^(.55){[f_1,f_2]}\ar[d]_\cong & \\
A\times(X+Y)\ar[r]_(.65)f & Z.
}\]
\section{Main definitions and first consequences}
Let $\E$ be a cartesian closed category, a {\em homotopy theory} on $\E$ is  a natural transformation $p:1_\E\Rightarrow\Pi_0:\E\rightarrow\E$ such that the functor $\Pi_0$ preserves finite products and the function $\E(1,p_X):\E(1,X)\Rightarrow\E(1,\Pi_0(X))$ is surjective for every $X\in\E$. In particular, it follows that the next diagram commutes:
\begin{equation}\label{E:homprod}
\begin{gathered}
\xymatrix{
X\times Y\ar[r]^{p_{X\times Y}}\ar[dr]_{p_X\times p_Y} & \Pi_0(X\times Y)\ar[d]^\cong \\
& \Pi_0(X)\times\Pi_0(Y).
}
\end{gathered}\end{equation}
 Two arrows $f,g:X\rightarrow Y$ in $\E$ are said to be \textit{homotopic} to each other if and only if their names $\text{`$f$',`$g$'}:1\rightarrow Y^X$ satisfy
\[
p_{Y^X}\circ\text{`$f$'}=p_{Y^X}\circ\text{`$g$'}.
\]
Denote this equivalence relation by $\sim$.

Define the \textit{homotopy category} $\E_p$ for $\E$ as follows: $\ob(\E_p):=\ob(\E)$ and
\begin{equation}
\E_p(X,Y) :=\E(X,Y)/{\sim}
\end{equation}
Composition is defined as the equivalence class of the composition of representatives. To see that this is indeed well defined, let $f,g\in\E(X,Y)$ with $f\sim g$ and $h,k\in\E(Y,Z)$ with $h\sim k$. Then, by \eqref{E:internalcomp} and \eqref{E:homprod}, the following diagram commutes:
\[\xymatrix{
1\ar@<.5ex>[rr]^(.4){\langle\text{`$h$'},\text{`$f$'}\rangle}\ar@<-.5ex>[rr]_(.4){\langle\text{`$k$'},\text{`$g$'}\rangle}\ar@/_1pc/@<.5ex>[drr]^(.55){\text{`$h\circ f$'}}\ar@/_1pc/@<-.5ex>[drr]_(.55){\text{`$k\circ g$'}} & & Z^Y\times Y^X\ar[rr]^(.45){p_{Z^Y}\times p_{Y^X}}\ar[d]^c & & \Pi_0(Z^Y)\times\Pi_0(Y^X)\ar[d]^{\Pi(c)} \\
& & Z^X\ar[rr]_{p_{Z^X}} & & \Pi_0(Z^X).
}\]
Hence $h\circ f\sim k\circ g$.

Notice that $\E(X,Y)/{\sim}\cong\E(1,Y^X)/{\sim'}$, with $\text{`$f$'}\sim'\text{`$g$'}$ if and only if $p_{Y^X}\circ\text{`$f$'}=p_{Y^X}\circ\text{`$g$'}$. So one has a bijective correspondence
\[\xymatrix{
\E(1,Y^X)/\sim'\ar[r]^\cong & \E(1,\Pi_0(Y^X)).
}\]
\begin{prop}\label{P:Epphinatural}
Let $\E$ be a cartesian closed category with a homotopy theory $p:1\Rightarrow\Pi_0$. For any arrow $\varphi:B\rightarrow A$  in $\E$, the natural transformation $\E_p(\varphi,-):\E_p(A,-)\Rightarrow\E_p(B,-)$ can be described as follows: Let $r:X\rightarrow Y$ be an arrow in $\E$, then
\[\xymatrix{
\E(1,\Pi_0(X^A))\ar[rr]^{\E(1,\Pi_0(X^\varphi))}\ar[d]_{\E(1,\Pi_0(r^A))} & & \E(1,\Pi_0(X^B))\ar[d]^{\E(1,\Pi_0(r^B))} &
u\ar@{|->}[r]\ar@{|->}[d] & \Pi_0(X^\varphi)\circ u\ar@{|->}[d]\\
\E(1,\Pi_0(Y^A))\ar[rr]_{\E(1,\Pi_0(Y^\varphi))} & & \E(1,\Pi_0(Y^B)) &
\Pi_0(r^A)\circ u\ar@{|->}[r] & \Pi_0(Y^\varphi\circ r^A)\circ u.
}\]
\end{prop}
\begin{proof} Since $Y^\varphi\circ r^A=r^B\circ X^\varphi$,
\[\xymatrix{
\E_p(A,X)\ar[r]^{\E_p(\varphi,X)}\ar[d]_{\E_p(A,r)} & \E_p(B,X)\ar[d]^{\E_p(B,r)} &
\text{`$f$'}\ar@{|->}[r]\ar@{|->}[d] & X^\varphi\circ\text{`$f$'}\ar@{|->}[d] \\
\E_p(A,Y)\ar[r]_{\E_p(\varphi,Y)} & \E_p(B,Y) &
r^A\circ\text{`$f$'}\ar@{|->}[r] & Y^\varphi\circ r^A\circ\text{`$f$'}\ar@{}|{=}[r] & r^B\circ X^\varphi\circ\text{`$f$'}
}\]
where $f\in\E_p(A,X)$. Note that the following commutative diagram:
\[\xymatrix{
1\ar[r]^{\text{`$f$'}} & X^A\ar[r]^{p_{X^A}}\ar[d]_{r^A} & \Pi_0(X^A)\ar[d]^{\Pi_0(r^A)} \\
& Y^A\ar[d]_{Y^\varphi}\ar[r]_{p_{Y^A}} & \Pi_0(Y^A)\ar[d]^{\Pi_0(Y^\varphi)} \\
& Y^B\ar[r]_{p_{Y^B}} & \Pi_0(Y^B).
}\]
The conclusion follows from the functoriality of $\Pi_0$.
\end{proof}

\begin{exem}
Let $\E$ be a cartesian closed category. The identity $1:1_\E\Rightarrow 1_\E$ is a homotopy theory.
\end{exem}
\begin{exem}
If $\E$ has an initial object and if $X\in\E$ has a point for every $X\neq 0$, then $!:1_\E\Rightarrow(-)^0$ is a homotopy theory.
\end{exem}
\begin{exem}
For a fixed object $A$,  $\sigma^A: 1\to (\,\cdot\,)^A$ might not be a homotopy theory when $A\not\cong1$.  
\end{exem}
\begin{exem}
Let $\E$ be a topos and let $j:\subob\rightarrow\subob$ be a local operator. Then, by \ref{T:modalityimplieshomotopy}, $q:1_\E\Rightarrow Q_j$ is a homotopy theory in $\E$ if $\E$ satisfies NS.
\end{exem}
\section{Categorical properties of the homotopy category}
Now that it has been established that $\E_p$ is a category, the main purpose of this section is to finish the proof of \hyperlink{T:TeorA}{Theorem A}: That $\E_p$ inherits the properties of being cartesian closed, extensive and distributive. 
\begin{prop}\label{P:ECCCimpliesEphasprod}
If $\E$ is a cartesian closed category with a homotopy theory, then $\E_p$ has products.
\end{prop}
\begin{proof}
Let $f,f':Z\rightarrow X$ and $g,g':Z\rightarrow Y$ be such that $f\sim f'$ and $g\sim g'$. To verify that $\langle f,g\rangle\sim\langle f',g'\rangle$, let $\zeta$ be the isomorphism $(X\times Y)^Z\cong X^Z\times Y^Z$. So
\[
\langle\text{`$f$'},\text{`$g$'}\rangle=\zeta\circ\text{`$\langle f,g\rangle$'}
\]
since in general, for an arrow $\langle r_1,r_2\rangle:A\times Z\rightarrow X\times Y$,
\[
\langle\widehat{r_1},\widehat{r_2}\rangle=\zeta\circ\widehat{\langle r_1,r_2\rangle}.
\]
Now, the following diagram commutes:
\[
\xymatrix{
& 1\ar@/_.7pc/@<-.5ex>[dl]_{\text{`$\langle f,g\rangle$'}}\ar@/_.7pc/@<.5ex>[dl]^{\text{`$\langle f',g'\rangle$'}} \ar@/^.7pc/@<-.5ex>[dr]_{\langle \text{`$f$'},\text{`$g$'}\rangle}\ar@/^.7pc/@<.5ex>[dr]^{\langle \text{`$f'$'},\text{`$g'$'}\rangle} && \\
Z^{X\times Y}\ar[rr]_\zeta\ar[d]_{p_{Z^{X\times Y}}}& & Z^X\times Z^Y\ar[rd]^{p_{Z^X}\times p_{Z^Y}}\ar[d]^{p_{Z^X\times Z^Y}}& \\
\Pi_0(Z^{X\times Y})\ar[rr]_{\Pi_0(\zeta)} & & \Pi_0(Z^X\times Z^Y)\ar[r]_(.45)\cong & \Pi_0(Z^X)\times\Pi_0(Z^Y).
}
\]
Therefore $\langle\text{`$f$'},\text{`$g$'}\rangle\sim\langle\text{`$f'$'},\text{`$g'$'}\rangle$. Hence $\text{`$\langle f,g\rangle$'}\sim\text{`$\langle f',g'\rangle$'}$.
\end{proof}
\begin{prop}\label{P:ECCCimpliesEpexp}
If $\E$ is a cartesian closed category with a homotopy theory, then $\E_p$ has exponentials.
\end{prop}
\begin{proof}
Let $f,f':X\times Z\rightarrow Y$ be such that $f\sim f'$. To verify that $\hat{f}\sim\widehat{f'}$, let $\xi$ be the isomorphism $Y^{X\times Z}\cong (Y^Z)^X$. So
\[
\text{`$\hat{f}$'}=\xi\circ\text{`$f$'}
\]
since in general, for an arrow $r:A\times X\times Z\rightarrow Y$,
\[
\widehat{\hat{r}}=\xi\circ\widehat{r},
\]
where on the left-hand side of the equation the inner $\hat{(-)}$ is with respect to $Z$ and the exterior one with respect to $X$, and on the right-hand side of the equation the $\widehat{(-)}$ is with respect to $X\times Z$. Therefore if $f\sim f'$, then $\hat{f}\sim\widehat{f'}$ since the following diagram commutes:
\[
\begin{gathered}[b]
\xymatrix{
& 1\ar@/_.7pc/@<-.5ex>[dl]_{\text{`$f$'}}\ar@/_.7pc/@<.5ex>[dl]^{\text{`$f'$'}} \ar@/^.7pc/@<-.5ex>[dr]_{\text{`$\hat{f}$'}}\ar@/^.7pc/@<.5ex>[dr]^{\text{`$\widehat{f'}$'}} & \\
Y^{X\times Z}\ar[rr]_\xi\ar[d]_{p_{Y^{X\times Z}}} & & (Y^Z)^X\ar[d]^{p_{(Y^Z)^X}} \\
\Pi_0(Y^{X\times Z})\ar[rr]_{\Pi_0(\xi)} & & \Pi_0((Y^Z)^X).
}\\[-\dp\strutbox]
\end{gathered}
\qedhere\]
\end{proof}

\begin{prop}\label{P:EdistimpliesEhso}
Let $\E$ be a cartesian closed category with a homotopy theory. If $\E$ has finite sums then $\E_p$ has finite sums and is therefore distributive. Explicitly, if $f_1,g_1:X\rightarrow Z$ and $f_2,g_2:Y\rightarrow Z$ are such that $f_1\sim g_1$ and $f_2\sim g_2$, then $[f_1,f_2]\sim[g_1,g_2]$.
\end{prop}
\begin{proof}
By \eqref{E:distributiveimpliestothesum},
\begin{align*}
\langle\text{`$f_1$'},\text{`$f_2$'}\rangle &=\alpha\circ\text{`$[f_1,f_2]$'}\\
\langle\text{`$g_1$'},\text{`$g_2$'}\rangle &=\alpha\circ\text{`$[g_1,g_2]$'}.
\end{align*}
Hence the following diagram commutes:
\[\xymatrix{
1\ar@<-.5ex>[d]_{\langle\text{`$f_1$'},\text{`$f_2$'}\rangle}\ar@<.5ex>[d]^{\langle\text{`$g_1$'},\text{`$g_2$'}\rangle} \ar@/^2pc/@<-.5ex>[drrr]_{\text{`$[f_1,f_2]$'}}\ar@/^2pc/@<.5ex>[drrr]^{\text{`$[g_1,g_2]$'}} &&& \\
Z^X\times Z^Y\ar[rrr]_{\alpha^{-1}}\ar[d]_{p_{Z^X}\times p_{Z^Y}}\ar[dr]^{p_{Z^X\times Z^Y}} &&& Z^{X+Y}\ar[d]^{p_{Z^{X+Y}}} \\
\Pi_0(Z^X)\times\Pi_0(Z^Y)\ar[r]_(.55)\cong &\Pi_0(Z^X\times Z^Y)\ar[rr]_{\Pi_0(\alpha^{-1})} && \Pi_0(Z^{X+Y}),
}\]
wherefore the conclusion follows.
\end{proof}

\begin{prop}\label{P:EextimpliesEpext}
Let $\E$ be a cartesian closed category with a homotopy theory. If $\E$ is extensive and distributive then so is $\E_p$.
\end{prop}
\begin{proof}
Suppose the canonical functor
\[
+:\E/X\times\E/Y\rightarrow\E/(X+Y)
\]
\[\xymatrix{
A\ar[d]_f & B\ar[d]_g\ar@{}[dr]|{\mapsto} & A+B\ar[d]^{f+g} \\
X, & Y & X+Y
}\]
is an equivalence for every pair of objects $X,Y\in\E$. So consider the analogous functor for $\E_p$ and name it $H(+)$. Hence, as $+$ is dense, so is $H(+)$. The fullness of $H(+)$ follows from that of $+$. Now, let $h,r:A\rightarrow A'$ and $k,s:B\rightarrow B'$ be arrows such that the following diagrams commute:
\[\xymatrix{
A+B\ar[rr]^{[h+k]}\ar[dr]_{[f+g]} & & A'+B'\ar[dl]^{[f'+g']} &
A+B\ar[rr]^{[r+s]}\ar[dr]_{[f+g]} & & A'+B'\ar[dl]^{[f'+g']}\\
& X+Y & &
& X+Y &
}\]
Suppose $[h+k]=[r+s]$. By \ref{P:EdistimpliesEhso}, 
\[
[i_X]\circ [h]=[i_X]\circ [r]\qquad\text{and}\qquad [i_Y]\circ [k]=[i_Y]\circ [s].
\]
As $\E_p$ is distributive (again by \ref{P:EdistimpliesEhso}), injections are monic (see \citet*[Proposition 3.3]{MR1201048}). Therefore $[h]=[r]$ and $[k]=[s]$. Hence $H(+)$ is faithful. That is, $H(+)$ is an equivalence of categories.
\end{proof}
\begin{proof}[Proof of \hyperlink{T:TeorA}{Theorem A}] This is the content of \ref{P:ECCCimpliesEphasprod}, \ref{P:ECCCimpliesEpexp}, \ref{P:EdistimpliesEhso}, and \ref{P:EextimpliesEpext}.
\end{proof}

\section{Explicit homotopies and contractibility I}
Two notions are fundamentally associated with the concept of homotopy theory. One is the concept of homotopy between maps and the other one is that of a space being contractible. At this generality not much more than the definition can be said, yet as advertised by the first parts of \hyperlink{T:TeorC}{Theorem C} and \hyperlink{T:TeorD}{Theorem D}, they are consistent with one's intuition. 

An {\em explicit homotopy} between two functions $f$ and $g$ in a cartesian closed category $\E$ with a homotopy theory $p:1\Rightarrow\Pi_0$ is an arrow $h:A\times X\rightarrow Y$ where $A$ is connected,  $\Pi_0(A)=1$, for which there are two points $a,b:1\rightarrow A$ such that
\[\xymatrix{
X\ar[r]_(.4){\langle a!,1\rangle}\ar@/^1pc/[rr]^f & A\times X\ar[r]_(.6)h & Y \\
X\ar[r]^(.4){\langle b!,1\rangle}\ar@/_1pc/[rr]_g & A\times X\ar[r]^(.6)h & Y.
}\]
An object $A$ is said to be {\em $p$-contractible} if it has a point $a:1\rightarrow A$ such that $a!\sim 1_A$.
\begin{teorema}\label{T:homotopyimplieshomotopicfunctions}
Let $\E$ be a cartesian closed category with a homotopy theory $p:1\Rightarrow\Pi_0$. If there is an explicit homotopy between $f$ and $g$, then $f\sim g$.
\end{teorema}
\begin{proof}
Let $h:A\times X\rightarrow Y$ be as required. The following diagram commutes:
\[\xymatrix{
1\times X\ar[d]_{a\times 1}\ar[r]^(.6){\pi_X} & X\ar[dl]^{\langle a!,1\rangle}\ar[dd]^f\\
A\times X\ar[d]_{\hat{h}\times 1}\ar[dr]^h & \\
Y^X\times X\ar[r]_(.6)\ev & Y.
}\]
Hence $\hat{h}a=\text{`$f$'}$. Similarly, $\hat{h}b=\text{`$g$'}$.

On the other hand, since $A$ is connected, $a\sim b$. Hence the following diagram commutes:
\[\begin{gathered}[b]
\xymatrix{
1\ar@<-.3ex>@/^2pc/[rr]_{\text{`$g$'}}\ar@<.7ex>@/^2pc/[rr]^{\text{`$f$'}}\ar@<.5ex>[r]^a\ar@<-.5ex>[r]_b & A\ar[r]^{\hat{h}}\ar[d]_{p_A} & Y^X\ar[d]^{p_{Y^X}} \\
& \Pi_0(A)\ar[r]^(.45){\Pi_0(\hat{h})} & \Pi_0(Y^X).
}\\[-\dp\strutbox]
\end{gathered}
\qedhere
\]
\end{proof}
\begin{teorema}\label{T:1TeoD}
Let $\E$ be a cartesian closed category with a homotopy theory $p:1\Rightarrow\Pi_0$, and let $A\in\E$. Then $A$ is $p$-contractible if and only if, for every object $X\in\E$, $\Pi_0(\sigma_X):\Pi_0(X)\rightarrow\Pi_0(X^A)$ is an isomorphism.
\end{teorema}
\begin{proof} By \eqref{E:Xtotheuniquesigma}, it is enough to verify the claim for $\Pi_0(X^!)$.  By \ref{P:Epphinatural}, $\Pi_0(X^!)$ is an isomorphism if and only if
\[\xymatrix{
\E_p(1,X)\ar[r]^{\E_p(!,X)} & \E_p(A,X)
}\]
is a natural isomorphism.  By Yoneda, this is equivalent to $A\cong_{\E_p} 1$, i.e. that there is a point $a:1\rightarrow A$ of $A$ such that $a!\sim 1_A$.
\end{proof}

\section{Reflectivity of the subcategory of decidables under homotopy}

Given a topos $\E$ precohesive over a topos $\mathcal S$, \citet{MR2369017} claims that $\mathcal S$ is also an exponential ideal of an associated homotopy category. From \cite{RS2023} it is known that a topos $\E$  that satisfies the NS is precohesive over a boolean base if and only if it also satisfies the WDQO and DSO postulates. From this, one can assume that $\mathcal S=\dec(\E)$. The purpose of this section is to exhibit $\dec(\E)$ as an exponential ideal of $\E_p$ when $\E$ is just assumed to satisfy  NS and WDQO, as advertised by \hyperlink{T:TeorB}{Theorem B}.

He suggests elsewhere in \cite{MR2369017} that having an explicit homotopy between maps should induce the same map between the objects of pieces. This is true when $\E$ is a topos that satisfies NS and WDQO.

\begin{prop}\label{P:Homtopicimpliessamepieces} Let $\E$ be a topos that satisfies NS and WDQO. Then $\Pi r=\Pi s$ for any two $r,s:X\rightarrow Y$ homotopic arrows in $\E$. If, furthermore, $Y$ is decidable, then $r=s$.
\end{prop}
\begin{proof}
 Notice that the following diagram commutes:
\[\xymatrix{
1\times X\ar[r]^{\pi_X}\ar@<-.5ex>[d]_{\text{`$r$'}\times 1}\ar@<.5ex>[d]^{\text{`$s$'}\times 1} & X\ar@<-.5ex>[d]_r\ar@<.5ex>[d]^s\ar[dr]^{p_X} &  \\
Y^X\times X\ar[dr]^(.69){p_{Y^X\times X}}\ar[r]_\ev\ar[d]_{p_{Y^X}\times p_X} &  Y\ar[dr]_(.38){p_Y} & \Pi X\ar@<-.5ex>[d]_{\Pi r}\ar@<.5ex>[d]^{\Pi s} \\
\Pi(Y^X)\times\Pi X\ar[r]_\cong & \Pi(Y^X\times X)\ar[r]_(.6){\Pi\ev} & \Pi Y.
}\]
Therefore, as $p_X\circ\pi_X$ is an epimorphism, $\Pi r=\Pi s$.  

Lastly, $p_Y$ is an isomorphism when $Y$ is decidable, thus finishing the proof.
\end{proof}
\begin{proof}[Proof of \hyperlink{T:TeorB}{Theorem B}] By \cite[Theorem B]{RS2023}), there is already an adjunction:
\[\xymatrix{
\E\ar@/_1pc/[d]_{\Pi}\ar@<-1ex>@{}[d]^{\dashv} \\
\dec(\E).\ar@/_1pc/[u]_{\mathcal{I}}
}\]
Its counit $\epsilon:\Pi\mathcal{I}\Rightarrow 1$ of $\Pi\dashv\mathcal{I}$ is also the counit for an adjunction $q_{!}\dashv q^\ast:\dec(\E)\rightarrow\E_p$. Indeed, the only thing to check is that if two arrows $f,g:X\rightarrow\mathcal{I}A$ are homotopic, then their corresponding arrows $\epsilon_A\circ\Pi f,\epsilon_A\circ\Pi g:\Pi X\rightarrow A$ under the adjunction $\Pi\dashv\mathcal{I}$ are the same. This follows from \ref{P:Homtopicimpliessamepieces}.
\end{proof}
\begin{obse}\label{O:counitequaltopinv}
Notice that since $\mathcal{I}$ is full and faithful, the counit $\epsilon$ of the adjunction $\dec(\E)(\Pi E,A)\cong\E(E,\mathcal{I}A)$ is the isomorphism 
\[
\epsilon_A=p_A^{-1}
\]
for every $A\in\dec(\E)$.
\end{obse}

\section{Decidability and the impossibility of motion}
\citet{MR3288694} proposes \eqref{E:nonbecoming} as part of the requirement for an object $T$ to be an unparameterized unit of time.  If  $\E$ is a topos satisfying NS and WDQO, this turns out to be satisfied by any connected exponent and any decidable base.   
\begin{teorema}\label{T:Decnonbecoming}
Let $\E$ be a topos satisfying NS and WDQO. If $T$ is a connected object with a point $0:1\to T$ and $A$ is decidable, then
\[
\ev^0_A:A^T\rightarrow A
\]
is an isomorphism.
\end{teorema}

\begin{proof}
It is clear that there are have natural isomorphisms
\[\xymatrix{
\E\rrtwocell^1_{-\times 1}{\varphi} & & \E
}\]
and
\[\xymatrix{
& \E\ar[dr]^{\Pi_0} & \\
\E\times\E\ar[ur]^\times\ar[dr]_{\Pi_0\times\Pi_0}\rrtwocell<\omit>{\psi} & & \E. \\
& \E\times\E\ar[ur]_\times &
}\]
Therefore the following horizontal composites are natural isomorphisms:
\begin{equation}\label{E:natiso1}
\begin{gathered}
\xymatrix{
\E^{op}\ar[r]^{\Pi_0} & \E^{op}\rrtwocell_1^{-\times 1}{\varphi} & & \E^{op}\ar[r]^{\E(-,A)} & \set
}
\end{gathered}
\end{equation}
and
\begin{equation}\label{E:natiso2}
\begin{gathered}
\xymatrix{
& \E^{op}\times\E^{op}\ar[dr]^\times & \\
\E^{op}\times\E^{op}\ar[ur]^{\Pi_0\times\Pi_0}\ar[dr]_\times\rrtwocell<\omit>{\psi} & & \E^{op}\ar[r]^{\E(-,A)} & \set. \\
& \E^{op}\ar[ur]_{\Pi_0} &
}\end{gathered}
\end{equation}

Now, suppose $A\in\dec(\E)$ and $X$ is any object of $\E$. The following are natural bijections:
\begin{align*}
\E(X,A^T) &\cong\E(X\times T,A) & \\
\E(X\times T,A) &\cong\dec(\E)(\Pi (X\times T),A). &\\
\dec(\E)(\Pi (X\times T),A)&\cong\E(\Pi_0(X\times T),A) & \\
\E(\Pi_0(X\times T),A) &\cong\E(\Pi_0 X\times\Pi_0 T,A) &\text{by }\eqref{E:natiso2}\\
\E(\Pi_0 X\times\Pi_0 T,A) &\cong\E(\Pi_0 X,A) &\text{by }\eqref{E:natiso1}\\
\E(\Pi_0 X,A) &\cong\dec(\E)(\Pi X,A) & \\
\dec(\E)(\Pi X,A) &\cong\E(X,A). &
\end{align*}
Therefore there is a natural isomorphism $\E(-,A^T)\cong\E(-,A)$. So, by Yoneda, $A^T\cong A$. 

More explicitly, consider arrows $f:X\rightarrow A^T$ and $g:X\rightarrow A$ in $\E$ related by this chain of isomorphisms. Successive corresponding arrows in the chain are represented by radial arrows proceeding counter-clockwise from $\hat f$ in the following diagram:
\begin{equation*}
\begin{xy}
(-14.7,20.225)*+{X\times T}="v3";
(-23.725,-7.725)*+{\Pi_0(X\times T)}="v1";%
(0,-25)*+{\Pi_0(X)\times\Pi_0(T)}="v0";
(23.725,-7.725)*+{\Pi_0(X)}="v2";%
(14.7,20.225)*+{X}="v4";
(0,0)*+{A}="v7";%
{\ar "v3"; "v1"}?*!/^3mm/{p}; 
{\ar "v1"; "v0"}?*!/^3mm/{\psi}; 
{\ar "v0"; "v2"}?*!/^3mm/{\pi_0}; 
{\ar "v4"; "v2"}?*!/_3mm/{p}; 
{\ar "v3"; "v4"}?*!/_3mm/{\pi_0}; 
{\ar "v0"; "v7"}?*!/_3mm/{}; 
{\ar "v1"; "v7"}?*!/_3mm/{}; 
{\ar "v2"; "v7"}?*!/_3mm/{}; 
{\ar "v3"; "v7"}?*!/_3mm/{\hat f}; 
{\ar "v4"; "v7"}?*!/^3mm/{g}; 
\end{xy}
\end{equation*}
Since the outer pentagon commutes, the whole diagram commutes with
\begin{equation}
\hat f=g\circ \pi_0,
\end{equation}
which by uniqueness must necessarily always hold.  Thus, to finish the proof, it remains to verify that 
\begin{equation}\label{E:Eval}
\ev_A^T=\widehat{1_{A^T}}=\ev_A^0\circ\pi_0.
\end{equation}
To this effect, recall that by \eqref{E:evt},  $\ev^0_A=\ev^T_A\circ\langle1,0!\rangle$, so that
\begin{equation}
\ev^0_A\circ\pi_0=\ev^T_A\circ\langle\pi_0,0!\rangle.
\end{equation}
Notice that 
\begin{align*}
\Pi_0(\langle\pi_0,0!\rangle)&=\psi^{-1}\circ\langle\Pi_0(\pi_0),\Pi_0(0!)\rangle\\
&=\psi^{-1}\circ\langle\Pi_0(\pi_0),!\rangle\\
&=\psi^{-1}\circ\langle\Pi_0(\pi_0),\Pi_0(\pi_1)\rangle\\
&=\psi^{-1}\circ\langle\pi_0\circ\psi,\pi_1\circ\psi\rangle\\
&=\psi^{-1}\circ\langle\pi_0,\pi_1\rangle\circ\psi\\
&=1.
\end{align*}
This implies that $\Pi_0(\ev_A^T)=\Pi_0(\ev_A^0\circ\pi_0)$ and---since $p_A$ is an isomorphism---\eqref{E:Eval} as well, which finishes the proof. 
\end{proof}

\section{Explicit homotopies and contractibility II}
This section continues the exploration on the concepts of explicit homotopy and of contractible space, first in the presence of NS+WDQO, and later under the full precohesion assumption NS+WDQP+DSO. In particular it completes the proofs of \hyperlink{T:TeorC}{Theorem C} and \hyperlink{T:TeorD}{Theorem D} through the following two theorems.

\begin{teorema}\label{T:homotopicimpliesahomotopy}
Let $\E$ be a topos satisfying NS and WDQO with its associated homotopy theory $p$. If $f\sim g:X\rightarrow Y$ then there is a bipointed connected object $A$ and an explicit homotopy $h:A\times X\rightarrow Y$ between them.
\end{teorema}
\begin{teorema}\label{T:contractibleiffAtotheX}
Let $\E$ be a topos satisfying NS, WDQO and DSO. Let $A\in\E$. Then 
\[
\Pi(A^X)=1
\]
for every $X\in\E$ if and only if there is a point $a:1\rightarrow A$ such that $a!\sim 1_A$. That is, $\Pi(A^X)=1$ for every $X\in\E$ if and only if $A$ is contractible under the associated homotopy theory $p$ for $\E$. 
\end{teorema}
\begin{proof}[Proof of \ref{T:homotopicimpliesahomotopy}]
Let $f,g:X\rightarrow Y$ be two homotopic arrows in $\E$. Let $K$ be the pullback
\[\xymatrix{
K\ar@{(>->}[rr]^j\ar[d] && Y^X\ar[d]^{p} \\
1\ar@{(>->}[rr]_{p(\text{`$f$'})} && \Pi_0(Y^X).
}\]
Thus $\Pi(K)=1$ by Proposition 2.1 and Theorem 2.4 in \cite{RS2023}. Now, let $a:1\rightarrow K$ and $b:1\rightarrow K$ be the corestrictions of $\text{`$f$'}:1\rightarrow Y^X$ and $\text{`$g$'}:1\rightarrow Y^X$ to $K$, resp. Let $h:K\times X\rightarrow Y$ be the following composite:
\[\xymatrix{
K\times X\ar@{(>->}[r]^{j\times 1} & Y^X\times X\ar[r]^(.63)\ev & Y.
}\]
The following diagram commutes:
\[\xymatrix{
X\ar[r]^{\langle a!,1\rangle}\ar[dr]_{\langle !,1\rangle}\ar@/_4.5pc/[drr]_1 & K\times X\ar@{(>->}[r]^{j\times 1}\ar@/^2pc/[rr]^h & Y^X\times X\ar[r]^(.63)\ev & Y \\
& 1\times X\ar[u]_{a\times 1}\ar[ur]_{\text{`$f$'}\times 1}\ar[r]_{\pi_X} & X\ar[ur]_f &
}\]
A similar diagram commutes for $g$.
\end{proof}
\begin{proof}[Proof of \ref{T:contractibleiffAtotheX}]
Suppose $\Pi(A^X)=1$ for every $X\in\E$. Hence, by definition,
\begin{align*}
\E_p(X,A) &\cong\E(1,\mathcal{I}\Pi(A^X))\\
&\cong\E(1,1)\\
&=\E_p(X,1).
\end{align*}
Since $\E(X,1)=1$, that isomorphism is natural. Therefore, by Yoneda, $A\cong_{\E_p} 1$. That is, there is a point $a:1\rightarrow A$ of $A$ such that $a!\sim 1_A$.

Conversely, suppose $A$ is contractible; that is, there is a point $a:1\rightarrow A$ in $A$ such that $a!\sim 1_A$. Hence
\[\xymatrix{
1\ar[r]^a & A\ar[r]^{\sigma_A} & A^X
}\]
is a point of $A^X$. Now, let $g:X\rightarrow A$ be an arbitrary arrow in $\E$. So, by hypothesis, $g\sim a!$. Hence $1=\E_p(X,A)=\E(1,\mathcal{I}\Pi(A^X))$. Therefore $\Pi(A^X)$ has just one necessarily dense point, so by Theorems C, D and 7.2 in \cite{RS2023}, $\Pi(A^X)=1$.
\end{proof}

\begin{cor}\label{C:contractibleiffasim1}
Let $\E$ be a nondegenerate topos satisfying NS, WDQO and DSO. Let $A$ be an object of $\E$ with a point. Then $\Pi(A^A)=1$ if and only if $A$ has a point $a:1\rightarrow A$ such that $a!\sim 1_A$.
\end{cor}
\begin{proof}
If $A$ has a point $a:1\rightarrow A$ such that $a!\sim 1_A$, $\Pi(A^A)=1$. Conversely, if $\Pi(A^A)=1$ and $A$ has a point $b:1\rightarrow A$, then $\E_p(A,A)=1$ and thus $b!\sim 1_A$.
\end{proof}

\begin{obse}
For $\E$ non degenerate, $0^0=\Pi(0^0)=1$ yet $0$ has no points.
\end{obse}

\section{Quality types vs. sufficient cohesion}

Let $\E$ be a topos precohesive over a Boolean base and satisfying NS. By Theorems C and D in \cite{RS2023}, this can be explicitly realized by the following string of adjoints:
\[\xymatrix{
\ar@<-3.5ex>@{}[rr]|{\dashv}\ar@<-3.5ex>@{}[rr]|(.26){\dashv}\ar@<-3.5ex>@{}[rr]|(.73){\dashv} & \E\ar@/^.7pc/[d]^{\Gamma}\ar@/_2.5pc/[d]_{\Pi} & \\
& \dec(\E)\ar@/^.7pc/[u]^{\mathcal{I}}\ar@/_2.5pc/[u]_{\Lambda} &
}\]
The following results within the section all have this assumption. Herein, two postulates are compared: Being a ``quality type'' and having sufficient cohesion.  The former means $\theta=p\circ\gamma$ is a natural isomorphism (see for example \cite[Lemma 9.2]{RS2023}); the latter that every object is a subobject of a Lawvere contractible. 
\begin{prop}
If $\E$ is a quality type and $f\sim g$, then $f=g$. 
\end{prop}
\begin{proof}
Let $f,g:X\rightarrow Y$ be two arrows in $\E$. The following diagram commutes:
\[\xymatrix{
1\ar@<.5ex>[r]^{\text{`$f$'}}\ar@<-.5ex>[r]_{\text{`$g$'}}\ar@/_1pc/@<.5ex>[dr]^{\bar{f}}\ar@/_1pc/@<-.5ex>[dr]_{\bar{g}} & Y^X\ar[r]^{p_{Y^X}} & \Pi(Y^X) \\
& \dso(Y^X)\ar[u]_{\gamma_{Y^X}} &
}\]
Since , $\bar{f}$ equals $\bar{g}$, and thus $f$ equals $g$. 
\end{proof}
\begin{teorema}\label{T:EnotQTimpliesSC}
If $\E$ is not a quality type, then it is sufficiently cohesive.
\end{teorema}
\begin{lema}\label{L:equivalencesforSC}
The following statements are equivalent:
\begin{enumerate}[(i)]
 \item $\E$ is sufficiently cohesive.
\item The subobject classifier is connected.
\item There is a bipointed connected object.
\end{enumerate}
\end{lema}
\begin{proof}
Proposition 4 in \citet{MR2369017} yields (i)$\Leftrightarrow$(ii). Now clearly, (ii)$\Rightarrow$(iii) since $\subob$ is a bipointed connected object.

(iii)$\Rightarrow$(ii). Let $K$ be a bipointed connected object. Let $a,b$ be those two different points of $K$. So the following diagram commutes:
\[\xymatrix{
1\ar[r]\ar[d]_a & 1\ar[d]^\top \\
K\ar[r]_{\chi_a} & \subob.
}\]
Now, since $b\neq a$ and $\subob$ is strictly bipointed because $\E$ satisfies NS, the following diagram commutes:
\[\xymatrix{
1\ar[r]\ar[d]_b & 1\ar[d]^\bot \\
K\ar[r]_{\chi_a} & \subob.
}\]
This yields a homotopy between $1_\subob$ and $\bot !$:
\[\xymatrix{
\subob\ar[r]_(.4){\langle a!,1\rangle}\ar@/^1.5pc/[rrr]^1 & K\times\subob\ar[r]_{\chi_a\times 1} & \subob\times\subob\ar[r]_(.6)\wedge & \subob \\
\subob\ar[r]^(.4){\langle b!,1\rangle}\ar@/_1.5pc/[rrr]_{\bot !} & K\times\subob\ar[r]^{\chi_a\times 1} & \subob\times\subob\ar[r]^(.6)\wedge & \subob.
}\]
Therefore, by \ref{T:contractibleiffAtotheX}, $\subob$ is contractible and, accordingly, connected.
\end{proof}
\begin{proof}[Proof of \ref{T:EnotQTimpliesSC}]
The ``points-to-pieces'' morphism $\theta_Y:\Gamma Y\rightarrow\Pi Y$ is monic in $\E$ if and only if it is monic in $\dec(\E)$ since $\mathcal{I}$ is right adjoint to $\Pi$. Therefore $\theta_Y$ is not monic in $\E$ if and only if it is not monic in $\dec(\E)$.

Suppose $\E$ is not a quality type; i.e., $\theta$ is not a natural isomorphism. Hence $\theta_X$ is not monic in $\E$ for some $X\in\E$ and, accordingly, not monic in $\dec(\E)$. 
So, as $\dec(\E)$ also satisfies NS, the local set theory $\teo(\dec(\E))$ is strongly witnessed and, therefore, complete (see Theorems 4.31 and 4.32 in \citet{MR972257}). Hence
\[
\vdash_{\teo(\dec(\E))}\theta_X\textit{ is injective}\quad\text{or}\quad\vdash_{\teo(\dec(\E))}\neg(\theta_X\textit{ is injective}).
\]
Therefore, by Proposition 3.33 in \citet{MR972257}, since $\theta_X$ is not monic in $\dec(\E)$,
\[
\nvdash_{\teo(\dec(\E))}\theta_X\textit{ is injective}.
\]
Therefore $\vdash_{\teo(\dec(\E))}\neg(\theta_X\textit{ is injective})$; equivalently,
\[
\vdash_{\teo(\dec(\E))}\neg\forall x\forall x'(x\in\Gamma X\wedge x'\in\Gamma X\wedge\theta_X(x)=\theta_X(x')\Rightarrow x=x').
\]
But $\dec(\E)$ is boolean since it is equivalent to $\E_{\neg\neg}$ (see \citet{MR0877866}). Hence
\[
\vdash_{\teo(\dec(\E))}\exists x\exists x'(x\in\Gamma X\wedge x'\in\Gamma X\wedge\theta_X(x)=\theta_X(x')\wedge\neg(x=x')).
\]
Now, by Theorem 4.31 in \citet{MR972257}, $\teo(\dec(\E))$ is witnessed because it is strongly witnessed; therefore, using this fact twice, there are closed terms $a,b$ of the appropriate type such that
\[
\vdash_{\teo(\dec(\E))}a\in\Gamma X\wedge b\in\Gamma X\wedge\theta_X(a)=\theta_X(b)\wedge\neg(a=b).
\]
So there are two different points $a,b:1\rightarrow\Gamma X$ of $\Gamma X$ such that $\theta_X\circ a=\theta_X\circ b$. Whence, as $\theta_X=p_X\circ\gamma_X$ by Lemma 9.2 in \cite{RS2023}, the corresponding points $\bar{a},\bar{b}:1\rightarrow X$ in $X$ satisfy $p_X\circ\bar{a}=p_X\circ\bar{b}$, and $p_X(\bar{a})=\theta_X(a)$. Let
\[
K:=p_X^{-1}(p_X(\bar{a})).
\]
Therefore, by Proposition 2.1 and Theorem 2.4 in \cite{RS2023}, $K$ is connected. Furthermore, $K$ is bipointed. So, by Lemma~\ref{L:equivalencesforSC}, $\E$ is sufficiently cohesive.
\end{proof}
\section{Connectedness and contractibility of double negation sheaves}
Taking advantage of the characterization given by \hyperlink{T:TeorC}{Theorem C} in the context of precohesion over a boolean base in the presence of the Nullstellensatz, it is possible to analyze the connectedness of $\neg\neg$-sheaves constructively in that context. Furthermore, it is also posible to produce explicit homotopies to prove that the sheafification of a contractible space is contractible too.  

\begin{teorema}
Let $\E$ be a topos that satisfies NS+WDQO+DSO. Every nonempty $\neg\neg$-sheaf is connected if and only if there is sufficient cohesion.
\end{teorema}
\begin{proof}
In light of \ref{L:equivalencesforSC}, suppose that $\E$ is sufficiently cohesive and let $K\in\E$ be a connected object with at least two different points $k_1,k_2$. Let $k'_1,k'_2$ be the corresponding points of $k_1,k_2$, resp. in $\Gamma K$. Let $G$ be a $\neg\neg$-sheaf such that there is an arrow $[g_1,g_2]:2\rightarrowtail\Pi G$. Since $\E$ satisfies NS and $p_G$ is surjective, let $g'_1,g'_2$ be two points in $G$ such that $p_G(g'_1)=g_1$ and $p_G(g'_2)=g_2$. Now, as $\Gamma K$ is decidable, define the following arrow $f:\Gamma K\rightarrow G$:
\[f:=
\begin{cases}
g'_1 &\text{if $x=k'_1$}\\
g'_2 &\text{if $x\neq k'_1$}.
\end{cases}
\]
Therefore, since $G$ is a $\neg\neg$-sheaf and $\gamma_K:\Gamma K\rightarrow K$ is monic $\neg\neg$-dense, there is a unique arrow $f':K\rightarrow G$ making the following diagram commute:
\[\xymatrix{
\Gamma K\ar[r]^(.55){\gamma_K}\ar[dr]_f & K\ar[d]^{f'} \\
& G.
}\]
Hence
\[
(f')^{-1}(g'_1)+(f')^{-1}(g'_2)=K,
\]
which, by Proposition 2.1 in \cite{RS2023}, is a contradiction since $\Pi K=1$. So $\Pi G$ has at most one point. If it has exactly one point, then $\Pi G=1$ since that point is $\neg\neg$-dense in $\Pi G$ (see Theorems C, D and 7.2 in \cite{RS2023}), and $G=0$ if and only if $\Pi G=0$ (see Lemma 2.2 in \cite{RS2023}).

Conversely, suppose every $\neg\neg$-sheaf is connected or 0. Consider $l_2:2\rightarrow\Lambda 2$. As 2 is decidable, 2 is $\neg\neg$-separated. Hence $l_2$ is monic $\neg\neg$-dense. That is, $\Lambda 2$ has exactly two points and is connected.
\end{proof}
Let $J:\E_{\neg\neg}\rightarrow\E$ be the inclusion functor, $L$ the $\neg\neg$-sheafification functor, and $l:1\Rightarrow JL$ the unit of the adjunction $L\dashv J$. 
\begin{teorema}\label{T:Lpreservescontractibility}
Let $\E$ be a topos that satisfies NS+WDQO+DSO. If $X\in\E$ is contractible, then so is its sheafification $LX$.
\end{teorema}
\begin{lema}\label{L:homotopylifts}
Let $\E$ be a topos that satisfies NS+DSO. Given arrows $h:K\times X\rightarrow Y$ and $f:Y\rightarrow LX$, there is a unique arrow $h':K\times LX\rightarrow LX$ making the following diagram commute:
\begin{equation}\label{E:Explicithomotopylifting}
\begin{gathered}
\xymatrix{
K\times X\ar[r]^(.6)h\ar[d]_{1\times l_X} & Y\ar[d]^{f} \\
K\times LX\ar@{-->}[r]^(.6){h'} & LX.
}
\end{gathered}
\end{equation}
\end{lema}
\begin{proof}
By the naturality of $l$, the following diagram commutes:
\[\xymatrix{
\Gamma X\ar[r]^{\gamma_X}\ar[d]_{l_{\Gamma X}} & X\ar[d]^{l_X} \\
L\Gamma X\ar[r]_{L\gamma_X} & LX.
}\]
Now, as $\Gamma X$ is decidable, it is also $\neg\neg$-separated and thus $l_{\Gamma X}$ is $\neg\neg$-dense monic. On the other hand, as $\gamma_X$ is $\neg\neg$-dense monic, $L\gamma_X$ is an isomorphism (see Proposition 5.26 in \citet{MR972257}). Therefore $l_X\circ\gamma_X$ is $\neg\neg$-dense monic (see Proposition A.4 in \cite{RS2023}). Hence, since $LX$ is a $\neg\neg$-sheaf and $(1\times l_X)\circ(\gamma_K\times\gamma_X):\Gamma K\times\Gamma X\rightarrow K\times LX$ is $\neg\neg$-dense monic, there is a unique $h':K\times LX\rightarrow LX$ making the following diagram commute:
\[\xymatrix{
\Gamma K\times\Gamma X\ar[d]_{\gamma_K\times\gamma_X} & \\
K\times X\ar[r]^(.6)h\ar[d]_{1\times l_X} & Y\ar[d]^{f} \\
K\times LX\ar@{-->}[r]_(.6){h'} & LX
}\]
Lastly, since $\gamma_K\times\gamma_X$ is $\neg\neg$-dense monic and $LX$ is a $\neg\neg$-sheaf, \eqref{E:Explicithomotopylifting} commutes.
\end{proof}
\begin{proof}[Proof of \ref{T:Lpreservescontractibility}]
Let $X\in\E$ be a contractible object. So, by \ref{T:homotopicimpliesahomotopy}, there is a connected object $K$ with two points $a,b$ and a homotopy $h:K\times X\rightarrow X$ such that
\[\xymatrix{
X\ar[r]_(.4){\langle a!,1\rangle}\ar@/^1pc/[rr]^1 & K\times X\ar[r]_(.6)h & X \\
X\ar[r]^(.4){\langle b!,1\rangle}\ar@/_1pc/[rr]_{c!} & K\times X\ar[r]^(.6)h & X.
}\]
for some point $c$ of $X$. Now, applying \ref{L:homotopylifts} to $h$ and $l_X$, there is a unique $h':K\times LX\rightarrow LX$ such that the following diagram commutes:
\begin{equation}\label{D:hlifting}
\begin{gathered}
\xymatrix{
K\times X\ar[r]^(.6)h\ar[d]_{1\times l_X} & X\ar[d]^{l_X} \\
K\times LX\ar[r]^(.6){h'} & LX.
}
\end{gathered}
\end{equation}
Whence the following diagram commutes:
\[\xymatrix{
X\ar[r]_(.4){\langle a!,1\rangle}\ar@/^1.3pc/[rr]^1\ar[d]_{l_X} & K\times X\ar[r]_(.6)h\ar[d]_{1\times l_X} & X\ar[d]^{l_X} \\
LX\ar[r]^(.4){\langle a!,1\rangle}\ar@/_1.3pc/[rr]_1 & K\times LX\ar[r]^(.6){h'} & LX.
}\]
By the universality of $l_X$,
\[\xymatrix{
LX\ar[r]^(.4){\langle a!,1\rangle}\ar@/_1.3pc/[rr]_1 & K\times LX\ar[r]^(.6){h'} & LX
}\]
commutes.

Now, again, by the commutativity of \eqref{D:hlifting}, the following diagram commutes:
\[\xymatrix{
X\ar[r]_(.4){\langle b!,1\rangle}\ar@/^1.3pc/[rr]^{c!}\ar[d]_{l_X} & K\times X\ar[r]_(.6)h\ar[d]_{1\times l_X} & X\ar[d]^{l_X} \\
LX\ar[r]^(.4){\langle b!,1\rangle}\ar@/_1.3pc/[rr]_{L(c!)=L(c)!} & K\times LX\ar[r]^(.6){h'} & LX.
}\]
By the universality of $l_X$,
\[\xymatrix{
LX\ar[r]^(.4){\langle b!,1\rangle}\ar@/_1.3pc/[rr]_{L(c)!} & K\times LX\ar[r]^(.6){h'} & LX
}\]
commutes. Therefore, by \ref{T:homotopyimplieshomotopicfunctions}, $L(c)!\sim 1_{LX}$. That is, $LX$ is contractible.
\end{proof}

\section{Additional examples of contractible spaces}
It was already remarked by \citet{MR2369017}, that any space with an action of a connected monoid with zero should be contractible. This would have the immediate application to any monoid with zero acting on itself. \hyperlink{T:TeorC}{Theorem C} and \hyperlink{T:TeorD}{Theorem D} provide a way to verifying this within their context of applicability (As was already the case for $\Omega$ once there is sufficient cohesion within the proof of \ref{L:equivalencesforSC}).

In fact, they also describe the behavior of Lawvere's unparameterized unit of time and its associated monoid $R$, which have been used as a basis for Synthetic Differential Geometry.  

\begin{teorema}\label{T:monoidwithzeroconnectediffcontractible}
Let $\E$ be a topos satisfying NS, WDQO and DSO. Let $X\in\E$ be a monoid with 0. Then $X$ is connected if and only if $X$ is contractible.
\end{teorema}
\begin{proof}
It is convenient to reason internally. Let $S$ be the local set theory on the local language of $\E$. As $X$ is a monoid with zero, there are closed terms $0$  and $1$ such that: $\vdash_S0\in X$; $x\in X\vdash_S x\cdot 0=0\cdot x=0$; $\vdash_S1\in X$; and $x\in X\vdash_Sx\cdot 1=1\cdot x=x$. The following diagrams commute:
\[\xymatrix{
X\ar[r]_(.4){\langle\bar{1}!,1\rangle}\ar@/^1pc/[rr]^{1_X} & X\times X\ar[r]_(.6)\bullet & X \\
R\ar[r]^(.4){\langle\bar{0}!,1\rangle}\ar@/_1pc/[rr]_{\bar{0}!} & X\times X\ar[r]^(.6)\bullet & X,
}\]
where $\bullet$ is the multiplication of the monoid $X$ and $\bar0,\bar1:1\to X$ are the points corresponding to the closed terms $0,1$. 
Thus $X$ is contractible if it is connected. 

Conversely, by \ref{T:contractibleiffAtotheX}, if $X$ is contractible, then $\Pi(X^1)=1$.
\end{proof}

\begin{cor}
Let $\E$ be a topos satisfying NS, WDQO and DSO. Let $T$ be an object with point $0:1\rightarrow T$. Let $R$ be the pullback object of $\ev^0_T:T^T\rightarrow T$ along 0. Then $R$ is connected if and only if it is contractible.
\end{cor}
\begin{proof}
The pullback of $\ev^0_T:T^T\rightarrow T$ along 0 is the same thing as calculating the pullback of $T^0$ along $\sigma_T\circ 0$, where $\sigma_T:T\rightarrow T^1$ is the transpose of $\pi_T:T\times 1\rightarrow T$. Reasoning internally, let $S$ be the local set theory on the local language of $\E$. So 
\[
R=\{\langle u,\ast\rangle:u\circ 0=\sigma_T\circ 0(\ast)\}
\]
in $\E$. More simply,
\[
R=\{u\in T^T:u\circ 0=0\}.
\]
Or equivalently,
\[
R=\{u\in T^T:\langle 0,0\rangle\in u\},
\]
taking $\ev^0_T$ and 0 instead of $T^0$ and $\sigma_T\circ 0$, respectively. It is clear that
\[
\vdash_S 1_T,0!\in R.
\]
So let $\overbar{1_T},\overbar{0!}:1\rightarrow R$ be such that $\overbar{1_T}(\ast)=1_T$ and $\overbar{0!}(\ast)=0!$, resp.

Now,
\begin{align*}
\langle u,v\rangle\in R &\vdash_Su\in R\wedge v\in R \\
&\vdash_S u\in T^T\wedge v\in T^T\wedge u\circ 0=0\wedge v\circ 0=0\\
&\vdash_S u\circ v\in T^T\wedge u\circ v\circ 0=0\\
&\vdash_S u\circ v\in R.
\end{align*}
That is, there is an $S$-function $(\langle u,v\rangle\mapsto u\circ v):R\times R\rightarrow R$. Call this function $\circ$. It is clear that
\[
\vdash_S\langle R,\circ,1_T,0!\rangle\textit{ is a monoid with zero}.
\]
Therefore, by \ref{T:monoidwithzeroconnectediffcontractible}, $R$ is connected if and only if it is contractible.
\end{proof}

\begin{obse}
For any topos $\E$ is a topos satisfying NS and WDQO epic images of connected objects are connected. Indeed, if $f:A\rightarrow B$ is an epic arrow in $\E$ with $\Pi_0(A)=1$, then, as $\Pi_0$ preserves epics since it is left adjoint, $\Pi_0 f$ is epic and monic and, accordingly, iso; that is, in that case $\Pi_0(B)\cong 1$.
\end{obse}

\begin{cor}
Let $\E$ be a topos satisfying NS, WDQO and DSO. If $T$ is an object of $\E$ with a $\neg\neg$-dense point $0:1\rightarrow T$, then if the pullback of 0 along $T^0$ is connected, so are $T^T$ and $T$, and, accordingly, $T$ is contractible.
\end{cor}
\begin{proof}
Consider the following pullback diagram:
\[\xymatrix{
R\ar@{(>->}[r]^i\ar[d] & T^T\ar[d]^{T^0} \\
1\ar@{(>->}[r]_0 & T. 
}\]
Since $0:1\rightarrowtail T$ is $\neg\neg$-dense, so is $i:R\rightarrowtail T^T$. Hence $\Pi(i)$ is epi. Therefore, if $\Pi(R)=1$ then $\Pi(i)$ is monic, and, accordingly, an isomorphism. Then $\Pi(T^T)=1$ and, by \ref{C:contractibleiffasim1}, $T$ is contractible. 
\end{proof}


\section{Final Thoughts}
Even though actual examples of CCC of topological spaces are frustratingly hard to come by, the results proved herein suggest that the topological notions postulated within the context of Axiomatic Cohesion (even without the Axiom of Continuity) do recover some deep truth about cohesion, albeit reasoning within a slightly unconventional logical framework---without the need to consider all-out modal logics.     

The notion of homotopy theory postulated in this report differs from the classical axioms in that no reference is made about its construction---it is thus a synthetic investigation. However, the fact that for any (Lawvere-Tierney) topology on a topos that further satisfies the Nullstellensatz there is a canonical homotopy theory associated with it suggests yet another connection to classical theories. This direction might still provide useful insights.

From the purely categorical viewpoint, and for the sake of completeness, the following definition is proposed. 

\begin{defii}[Morphisms between homotopy theories]Given two $(\E,p),(\E',p')$ homotopy theories, a morphism $(\E,p)\rightarrow(\E',p')$ between them is a pair $(F,q)$ where $F:\E\rightarrow\E'$ is a cartesian closed functor and $q:F\Pi_0\Rightarrow\Pi_0'F$ a natural transformation such that
\[\xymatrix{
\E\rrtwocell^1_{\Pi_0}{p}\ar[d]_F & & \E\ar[d]^F\ar@{}|{=}[r] &
\E\ar[r]^F & \E'\rrtwocell^1_{\Pi_0'}{p'} & & \E'. \\
\E'\ar[rr]_{\Pi_0'}\rrtwocell<\omit>{<-2.3>q} & & \E'
}\]
That is, $q\cdot Fp=p'F$. It is clear that the identity morphism is $(1,1)$.
\end{defii}

\appendix

\renewcommand\thesection{\Alph{section}}

\section{Examples of the Nullstellensatz in precohesive toposes}
Examples of toposes abound that do not satisfy the Nullstellensatz postulate. However, Theorem \ref{T:Cwith1npoints} below, together with the following two propositions show that pre-sheaf and Grothendieck toposes necessarily satisfy it if they are precohesive over sets with the canonical four-adjoint string. 
\begin{prop}[Proposition 4.1 in \citet{MR3263283}]
Let $C$ be a small category whose idempotents split. The canonical $p:\hat{C}\rightarrow\set$ is pre-cohesive if and only if $C$ has a terminal object and every object of $C$ has a point.
\end{prop}
\begin{prop}[Proposition 1.4 in \citet{MR2805745}]\label{P:pJNS}
Let $p:\E\rightarrow\set$ be a bounded geometric morphism. If $p$ is precohesive, then $\E$ has a site of definition $\sh(C,J)$ of $\E$ which has a terminal object and all of its objects have a point.
\end{prop}
\begin{teorema}\label{T:Cwith1npoints}
Let $C$ be a small category with initial object and such that every object $c\in C$ has a copoint, an arrow $c\rightarrow 0$. Then $\set^C$ satisfies NS.
\end{teorema}
\begin{proof}
Let $F\in\set^C$ be different from $0$. So there is an object $d\in C$ such that $Fd\neq\emptyset$. Since $d$ has a copoint $q_d:d\rightarrow 0$, then $F0\neq\emptyset$. Now, let $c\in C$. Via $!:0\rightarrow c$, it follows that $Fc\neq\emptyset$.

Now, let $x\in F0$. Given $c\in C$, define $p_c:1\rightarrow Fc$ as
\[
p_c(\ast):=F(!)(x),
\]
where $!:0\rightarrow c$. Hence $p$ is a natural transformation $1\Rightarrow F$. Therefore $\set^C$ satisfies NS.
\end{proof}
\begin{cor}\label{C:NSrules}
Let $p:\E\rightarrow\set$ be a geometric morphism.  If $p$ is the canonical four-adjoint string from $\E=\hat{C}$ for a small category $C$, or if it is bounded, then $\E$ satisfies the Nullstellensatz if $p$ is precohesive. Furthermore, in these cases, $\dec(\E)\simeq\set$ and the discrete objects of $\E$ are its decidable objects.
\end{cor}
\begin{proof}
In the first case, since for any small category $C$, its Karoubi envelope $K(C)$ has split idempotents, and $\widehat{C}$ is equivalent to $\widehat{K(C)}$,  by \ref{T:Cwith1npoints}, $\hat{C}$ satisfies NS. In the second case let $C$ be the site of definition $\sh(C,J)$ of $\E$. Now, by \ref{P:pJNS}, $\hat{C}$ satisfies NS. By fullness, so do $\sh(C,J)$ and, thus, $\E$.

Finally, by Theorem D in \cite{RS2023}, the conclusion follows.
\end{proof}

\section{Lawvere-Tierney topologies and their homotopy theory}\label{A:EQM}
The main goal of this section is to see that a topology induces a homotopy theory on a topos. This is done first by recalling how equivalence relations induce equivalence classes, which in turn induce quotient maps, in toposes. For topologies, these quotients objects are seen to be functorial---and finite-product-preserving---and their projections natural. 

Throughout this section, the local language and local set theory of a topos $\E$ will be used extensively. As discussed by the authors thoroughly in \cite{RS2021}, local languages in general don't have enough function symbols for all the syntactic functions that can arise from internal constructions. The local language of a topos does not have this problem. Yet in this section, several constructions are given as descriptions and thus it is advantageous to work as if no function symbol is present. The way this is dealt with is as in basic set theoretical considerations where a function $f:X\to Y$ is defined by way of its graph $|f|\subseteq X\times Y$ (see \cite{MR972257}). The usual requirements are expected to be satisfied by it. There is a special case given by any term $\tau$ with free variables $x_1,\ldots, x_n$. In that case, $(\langle x_1,\ldots, x_n\rangle\mapsto \tau)$ denotes the function the graph of which is 
\[
\{\langle x_1,\ldots, x_n, y\rangle|y=\tau(x_1,\ldots, x_n)\},
\]
restricted to the appropriate objects. As a particular case, when the term has no free variables, i.e. a closed term, this yields a global element.  

The internal language equivalent of a topology is a modality. This is an assignment $\mu$ to each formula $\alpha$, i.e. a term of type $\Omega$, another formula $\mu(\alpha)$ such that: (1) $\alpha\vdash_S\mu(\alpha)$; (2) if $\alpha\vdash_S\beta$, then $\mu(\alpha)\vdash_S\mu(\beta)$; and (3) $\mu(\mu(\alpha))\vdash_S\mu(\alpha)$. From these one also obtains (4) $\vdash_S\mu(\alpha\wedge\beta)=\mu(\alpha)\wedge\mu(\beta)$, which in turn yields that
\begin{equation}\label{E:Gammamuimplication}
\mu(\alpha),\Gamma\vdash_S\mu(\beta),
\end{equation}
whenever $\alpha,\Gamma\vdash_S\beta$.

Equivalence relations are defined word by word as in classical set theory: subobjects of the binary product that satisfy reflexivity, symmetry and transitivity. Thus, let $R$ be an equivalence relation on an object $X$. Define
\begin{align}
C(x) &:=\{y\in X:\langle x,y\rangle\in R\}\notag\\
Q_R  &:=\{t:\exists x\in X. t=C(x)\}\notag.
\end{align}
It is readily verified that there is a function $q_X:=(x\mapsto C(x)):X\rightarrow Q_R$.
Indeed, 
\[
x\in X \vdash_S x\in X\wedge C(x)=C(x)\vdash_S\exists z\in X.C(x)=C(z)\vdash_SC(x)\in Q_R.
\]

It is also straightforward to verify the following entailments. Let $f:X\rightarrow Y$ and $g:Y\rightarrow Z$.
\begin{equation}\label{E:equivalenceonrelationsnfunctions}
\vdash_S\langle x,y\rangle\in\lvert f\rvert\Leftrightarrow\langle x',y\rangle\in\lvert f\rvert)\Leftrightarrow(\langle x,y\rangle\in\lvert f\rvert\wedge\langle x',y'\rangle\in\lvert f\rvert\Rightarrow y=y').
\end{equation}
and
\begin{equation}\label{E:precomposingwitharepresentable}
\vdash_S\lvert g\circ f\rvert=\{\langle x,z\rangle\in X\times Z:\langle\tau(x),z\rangle\in\lvert g\rvert\},
\end{equation}
when $f=(x\mapsto\tau)$ for some term $\tau$. And, equivalently,
\begin{equation}\label{E:pwp}
x\in X\vdash_S\langle \tau(x),z\rangle\in\lvert g\rvert\Leftrightarrow\langle x,z\rangle\in\lvert f\rvert.
\end{equation}

The following theorem provides the expected universal property of quotients. 
\begin{teorema}\label{T:universalityofQnegneg}
Let $R$ be an equivalence relation on $X$. Let $f:X\rightarrow Y$ be a function such that
\begin{equation}\label{E:functionsthatpreserveR}
\langle x,x'\rangle\in R\vdash_S\langle x,y\rangle\in\lvert f\rvert\Leftrightarrow\langle x',y\rangle\in\lvert f\rvert.
\end{equation}
Then, there is a unique function $\bar{f}:Q_R\rightarrow Y$ such that the following diagram commutes:
\[\xymatrix{
X\ar[r]^{q_X}\ar[dr]_f & Q_R\ar@{-->}[d]^{\bar{f}} \\
& Y.
}\]
From which it follows that $q_X$ is epic.
\end{teorema}
Notice that, since $\langle x,x'\rangle\in R\vdash_S C(x)=C(x')$, any composite $g\circ q_X$ satisfies \eqref{E:functionsthatpreserveR}, by \eqref{E:pwp} and \eqref{E:equivalenceonrelationsnfunctions}.

\begin{proof}[Proof of \ref{T:universalityofQnegneg}]
Define
\[\lvert\bar{f}\rvert:=\{\langle t,y\rangle:\exists x\in X.t=C(x)\wedge\langle x,y\rangle\in\lvert f\rvert\}.\]
Then,
\begin{align}
x\in X\wedge t=C(x) &\vdash_S(\exists!y.\langle x,y\rangle\in\lvert f\rvert)\wedge t=C(x)\notag\\
&\vdash_S\exists!y.\langle x,y\rangle\in\lvert f\rvert\wedge t=C(x)\notag\\
&\vdash_S\exists x\exists!y.\langle x,y\rangle\in\lvert f\rvert\wedge t=C(x)\notag\\
&\vdash_S\exists!y\exists x.\langle x,y\rangle\in\lvert f\rvert\wedge t=C(x)\notag.
\end{align}
Hence, $\exists x.x\in X\wedge t=C(x)\vdash_S\exists!y\exists x.\langle x,y\rangle\in\lvert f\rvert\wedge t=C(x)$. Therefore,
\[t\in Q_R\vdash_S\exists!y\exists x.\langle x,y\rangle\in\lvert f\rvert\wedge t=C(x),\]
as required for a function $\bar{f}:Q_R\rightarrow Y$.

To see that $\bar{f}\circ q_X=f$, one must prove that 
\[
\vdash_S|\bar{f}\circ q_X|=|f|.
\]
To this effect, observe that
\begin{align}
t=C(x),t=C(x'),\langle x',y\rangle\in\lvert f\rvert &\vdash_S\langle x,x'\rangle\in R\wedge\langle x',y\rangle\in\lvert f\rvert\notag\\
&\vdash_S\langle x,y\rangle\in\lvert f\rvert\notag.
\end{align}
Hence
\[t=C(x),\exists x'(t=C(x')\wedge\langle x',y\rangle\in\lvert f\rvert)\vdash_S\langle x,y\rangle\in\lvert f\rvert.\]
Whence
\[\langle x,t\rangle\in q_X\wedge\langle t,y\rangle\in\bar{f}\vdash_S\langle x,y\rangle\in\lvert f\rvert.\]
and thus
\[
\vdash_S|\bar{f}\circ q_X|\subseteq |f|.
\]
Conversely, $\langle x,y\rangle\in\lvert f\rvert\vdash_S\langle x,C(x)\rangle\in\lvert q_X\rvert$. On other hand,
\begin{align}
\langle x,y\rangle\in\lvert f\rvert &\vdash_SC(x)=C(x)\wedge\langle x,y\rangle\in\lvert f\rvert\notag\\
&\vdash_S\exists x'\in X(C(x)=C(x')\wedge\langle x',y\rangle\in\lvert f\rvert)\notag\\
&\vdash_S\langle C(x),y\rangle\in\lvert\bar{f}\rvert\notag.
\end{align}
Hence $\langle x,y\rangle\in\lvert f\rvert\vdash_S\langle x,C(x)\rangle\in\lvert q_X\rvert\wedge\langle C(x),y\rangle\in\lvert\bar{f}\rvert$. So
\[\langle x,y\rangle\in\lvert f\rvert\vdash_S\exists t.\langle x,t\rangle\in\lvert q_X\rvert\wedge\langle t,y\rangle\in\lvert\bar{f}\rvert.\]
Therefore 
\[\vdash_S|f|\subseteq |\bar{f}\circ q_X|
\]
and thus $\bar{f}\circ q_X=f$.

To see that $\bar{f}$ is unique, let $g:Q_R\rightarrow Y$ be such $g\circ q_X=f$. By \eqref{E:pwp},
\[
x\in X\wedge t=C(x),\langle t,y\rangle\in\lvert g\rvert \vdash_S\langle C(x),y\rangle\in\lvert g\rvert
\vdash_S\langle x,y\rangle\in\lvert f\rvert
\]
Thus,
\begin{align}
x\in X\wedge t=C(x),\langle t,y\rangle\in\lvert g\rvert &\vdash_St=C(x)\wedge\langle x,y\rangle\in\lvert f\rvert\notag\\
&\vdash_S\exists x\in X.t=C(x)\wedge\langle x,y\rangle\in\lvert f\rvert\notag\\
&\vdash_S\langle t,y\rangle\in\lvert\bar{f}\rvert\notag.
\end{align}
Hence, $\exists x\in X(t=C(x)),\langle t,y\rangle\in\lvert g\rvert\vdash_S\langle t,y\rangle\in\lvert\bar{f}\rvert$. So, as
\[
\langle t,y\rangle\in\lvert g\rvert \vdash_St\in Q_R\vdash_S\exists x\in X(t=C(x)),
\]
it follows that 
\[\langle t,y\rangle\in\lvert g\rvert\vdash_S\langle t,y\rangle\in\lvert\bar{f}\rvert.\]
Uniqueness thus follows by reversing the roles of $\bar f$ and $g$.\end{proof}
\begin{teorema}\label{T:modalityimplieshomotopy} For any modality $\mu$ the relation given by
\[R_\mu(X):=\{\langle x,y\rangle\in X\times X:\mu(x=y)\}\]
is an equivalence relation the quotient of which $Q_\mu(X)$ is functorial and preserves finite products.  Moreover, the quotient map $q^X_\mu$ is natural.
\end{teorema}
\begin{obse}
Let $\D$ be the full subcategory of objects  $X$ for which $q^X_\mu$ is an isomorphism and let $\iota:\D\rightarrow\E$ be the inclusion functor. It can be further seen that, by correstricting $Q_\mu$ to $\D$, 
\[Q_\mu\dashv\iota,\]
and that all objects of $\D$ are $\mu$-separated.   
\end{obse}
\begin{proof}[Proof of \ref{T:modalityimplieshomotopy}]
It is indeed an equivalence relation on $X$: 
\begin{enumerate}[(a)]
\item Reflexivity. As $x=x\vdash_S\mu(x=x)$ and $\vdash_Sx=x$, it follows that $\vdash_S\mu(x=x)$. Hence
\[x\in X\vdash_S\mu(x=x).\]
\item Symmetry. As $x=y\vdash_Sy=x$ and $\mu$ is a modality in $S$, thus
\[\mu(x=y)\vdash_S\mu(y=x).\]
\item Transitivity. As $x=y\wedge y=z\vdash_Sx=z$ and $\mu$ is a modality in $S$,
\[\mu(x=y\wedge y=z)\vdash_S\mu(x=z).\]
As $\vdash_S\mu(x=y)\wedge\mu(y=z)=\mu(x=y\wedge y=z)$, then
\[\mu(x=y)\wedge\mu(y=z)\vdash_S\mu(x=z).\]
\end{enumerate}

Let $f:X\rightarrow Y$ be an function. It is needed to prove that $q^Y_\mu\circ f$ satisfies \eqref{E:functionsthatpreserveR}. By \eqref{E:equivalenceonrelationsnfunctions}, it is enough to verify that
\begin{equation}\label{E:Qnegnegfunctorness}
\langle x,x'\rangle\in R_\mu(X)\vdash_S\langle x,z\rangle\in\lvert q_Y\circ f\rvert\wedge\langle x',z'\rangle\in\lvert q_Y\circ f\rvert\Rightarrow z=z'.
\end{equation}

Now, clearly
\[
x=x',\langle x,y\rangle\in\lvert f\rvert,\langle x',y'\rangle\in\lvert f\rvert\vdash_Sy=y'.
\]
Hence, by \eqref{E:Gammamuimplication},
\begin{equation}\label{E:doublenegfunctiondoubleneg}
\mu(x=x'),\langle x,y\rangle\in\lvert f\rvert,\langle x',y'\rangle\in\lvert f\rvert\vdash_S\mu(y=y').
\end{equation}
On the other hand, by definition of $q$, 
\[\langle y,z\rangle\in\lvert q_Y\rvert\vdash_Sz=C(y).\]
Applying this twice and \eqref{E:doublenegfunctiondoubleneg}, 
\begin{align}
\langle x,x'\rangle\in R_\mu(X),\langle x,y\rangle &\in\lvert f\rvert\wedge\langle y,z\rangle\in\lvert q_Y\rvert,\langle x',y'\rangle\in\lvert f\rvert\wedge\langle y',z'\rangle\in\lvert q_Y\rvert\notag\\
&\vdash_Sz=C(y)\wedge z'=C(y')\wedge C(y)=C(y')\notag\\
&\vdash_Sz=z'\notag.
\end{align}
Whence \eqref{E:Qnegnegfunctorness} is proved.

Then, by \ref{T:universalityofQnegneg}, there is a unique function $Q_\mu^f:Q_\mu^X\rightarrow Q_\mu^Y$ such that the following diagram commutes:
\[\xymatrix{
X\ar[r]^f\ar[d]_{q^X_\mu} & Y\ar[d]^{q^Y_\mu} \\
Q_\mu^X\ar@{-->}[r]_{Q_\mu^f} & Q_\mu^Y,
}\] 
which established both the functoriality of $Q_\mu$ and the naturality of $q_\mu$. 
It remains to see that it that preserves finite products. Since $\mu$ is a modality and by \eqref{E:Gammamuimplication},
\[
x,x',y,y'\in X\vdash_S\mu(\langle x,y\rangle=\langle x',y'\rangle)\Leftrightarrow(\mu(x=x')\wedge\mu(y=y')).
\]
So
\[
x,x',y,y'\in X\vdash_S C(\langle x,y\rangle)=C(\langle x',y'\rangle) \Leftrightarrow C(x)=C(x')\wedge C(y)=C(y')\]
and thus
\[
x,x',y,y'\in X\vdash_S C(\langle x,y\rangle)=C(\langle x',y'\rangle) \Leftrightarrow q_X\times q_Y(\langle x,y\rangle)=q_X\times q_Y(\langle x',y'\rangle).
\]

Hence, by \ref{T:universalityofQnegneg}, there is a unique function $r:Q_\mu(X\times Y)\rightarrow Q_\mu X\times Q_\mu Y$ such that the following diagram commutes:
\[\xymatrix{
X\times Y\ar[r]^(.45){q_{X\times Y}}\ar[dr]_{q_X\times q_Y} & Q_\mu(X\times Y)\ar[d]^r \\
& Q_\mu X\times Q_\mu Y
}\]
Since $q_X\times q_Y$ is epic, $r$ is the promised isomorphism.
\end{proof}


\bibliographystyle{plainnat}
\bibliography{../ref}

\end{document}